\documentclass[12pt]{article}

\usepackage{amsmath,amssymb,amsthm,fullpage}

\begin{document}

\newtheorem{theorem}{Theorem}[section]
\newtheorem{definition}[theorem]{Definition}
\newtheorem{corollary}[theorem]{Corollary}
\newtheorem{lemma}[theorem]{Lemma}
\newtheorem{remark}[theorem]{Remark}
\newtheorem{proposition}[theorem]{Proposition}
\newtheorem{example}[theorem]{Example}
\newtheorem{conjecture}[theorem]{Conjecture}

\numberwithin{equation}{section}

\def\cZ{\mathcal Z}

\def\C{{\mathbb C}}
\def\F{{\mathbb F}}
\def\K{{\mathbb K}}
\def\L{{\mathbb L}}
\def\N{{\mathbb N}}
\def\Q{{\mathbb Q}}
\def\R{{\mathbb R}}
\def\Z{{\mathbb Z}}
\def\E{{\mathbb E}}
\def\T{{\mathbb T}}
\def\P{{\mathbb P}}


\def\eps{\varepsilon}
\def\mand{\qquad\mbox{and}\qquad}
\def\\{\cr}
\def\({\left(}
\def\){\right)}
\def\[{\left[}
\def\]{\right]}
\def\<{\langle}
\def\>{\rangle}
\def\fl#1{\left\lfloor#1\right\rfloor}
\def\rf#1{\left\lceil#1\right\rceil}
\def\le{\leqslant}
\def\ge{\geqslant}
\def\ds{\displaystyle}

\def\xxx{\vskip5pt\hrule\vskip5pt}
\def\yyy{\vskip5pt\hrule\vskip2pt\hrule\vskip5pt}
\def\imhere{ \xxx\centerline{\sc I'm here}\xxx }

\newcommand{\comm}[1]{\marginpar{
\vskip-\baselineskip \raggedright\footnotesize
\itshape\hrule\smallskip#1\par\smallskip\hrule}}


\title{\bf Zeros of Polynomials with Random Coefficients}

\author{
{\sc Igor E.~Pritsker} \\
{Department of Mathematics, Oklahoma State University} \\
{Stilwater, OK 74078 USA} \\
{\tt igor@math.okstate.edu}
\and
{\sc Aaron M.~Yeager} \\
{Department of Mathematics, Oklahoma State University} \\
{Stillwater, OK 74078 USA} \\
{\tt aaron.yeager@math.okstate.edu}}

\maketitle

\begin{abstract}
Zeros of many ensembles of polynomials with random coefficients are asymptotically equidistributed near the unit circumference. We give quantitative estimates for such equidistribution in terms of the expected discrepancy and expected number of roots in various sets. This is done for polynomials with coefficients that may be dependent, and need not have identical distributions. We also study random polynomials spanned by various deterministic bases.
\end{abstract}


\section{Introduction}

Zeros of polynomials of the form $P_n(z)=\sum_{k=0}^{n} A_k z^k,$ where $\{A_n\}_{k=0}^n$ are random coefficients,  have been studied by Bloch and P\'olya, Littlewood and Offord, Erd\H{o}s and Offord, Kac, Rice, Hammersley, Shparo and Shur, Arnold, and many other authors. The early history of the subject with numerous references is summarized in the books \cite{BR, Fa}. It is now well known that, under mild conditions on the probability distribution of the coefficients, the majority of zeros of these polynomials is accumulating near the unit circumference, and they are also equidistributed in the angular sense.  Introducing modern terminology, we call a collection of random polynomials $P_n(z)=\sum_{k=0}^n A_k z^k,\ n\in\N,$ the ensemble of \emph{Kac polynomials}.
Let $Z(P_n)=\{Z_1,Z_2,\ldots , Z_n\}$ be the set of complex zeros of a polynomial $P_n$ of degree $n$.  These zeros $\{Z_k\}_{k=1}^n$ give rise to the \emph{zero counting measure}
$$\tau_n=\frac{1}{n}\sum_{k=1}^n \delta_{Z_k},$$
which is a random unit Borel measure in $\C.$ The fact of equidistribution for the zeros of random polynomials can now be expressed via the convergence of $\tau_n$ in the $\text{weak}^*$ topology to the the normalized arclength measure $\mu_{\T}$ on the unit circumference, where $d\mu_{\T}(e^{it}):=dt/(2\pi).$ Namely, we have that $\tau_n \stackrel{*}{\rightarrow} \mu_{\T}$ with probability 1 (abbreviated as a.s. or almost surely). More recent papers on zeros of random polynomials include \cite{HN,IZ,IZa,KZ}.
In particular, Ibragimov and Zaporozhets \cite{IZa} proved that if the coefficients are independent and identically distributed, then the condition $\E[\log^+|A_0|]<\infty$ is necessary and sufficient for $\tau_n \stackrel{*}{\rightarrow} \mu_{\T}$ almost surely. Here, $\E[X]$ denotes the expectation of a random variable $X$.

Our goal is to provide estimates on the expected rate of such convergence. A standard way to study the deviation of $\tau_n$ from $\mu_{\T}$ is to consider the discrepancy of these measures in the annular sectors of the form
$$ A_r(\alpha,\beta)=\{z\in \C :r<|z|<1/r, \ \alpha \leq \text{arg} \ z <\beta \}, \quad 0<r<1.$$
Such estimates were recently provided by Pritsker and Sola \cite{PS} using the largest order statistic $Y_n=\max_{k=0,\ldots, n}|A_k|$. The results of \cite{PS} require the coefficients $\{A_k\}_{k=0}^n$ be independent and identically distributed (iid) complex random variables having absolutely continuous distribution with respect to the area measure. This assumption excluded many important distributions such as discrete ones, in particular. We remove many unnecessary restrictions in this paper, and generalize the results of \cite{PS} in several directions. Section 2 develops essentially the same theory as in \cite{PS} (but uses a different approach) for the case of coefficients that are neither independent nor identically distributed, and whose distributions only satisfy certain uniform bounds for the fractional and logarithmic moments. We also consider random polynomials spanned by general bases in Section 3, which includes random orthogonal polynomials on the unit circle and the unit disk. Section 4 shows how one can handle the discrete random coefficients by methods involving the highest order statistic $Y_n$, augmenting the ideas of \cite{PS}. We further develop the highest order statistic approach to the case of dependent coefficients in Section 5, under the assumption that the coefficients satisfy uniform bounds on the first two moments. All proofs are contained in Section 6.

\section{Expected Number of Zeros of Random Polynomials}

Let $A_k,\ k=0,1,2,\ldots,$ be complex valued random variables that are not necessarily independent nor identically distributed, and let $\|P_n\|_{\infty}=\sup_{\mathbb{T}}|P_n|$. We study the expected deviation of the normalized number of zeros from $\mu_{\T}$ in annular sectors, which is often referred to as discrepancy between those measures.

\begin{theorem} \label{thm2.1}
Suppose that the coefficients of $P_n(z)=\sum_{k=0}^n A_k z^k$ are complex random variables that satisfy:
\begin{enumerate}
\item $\E[|A_k|^t]<\infty,\ k=0,\ldots,n,$ for a fixed $t>0$
\item $\E[\log|A_0|] > -\infty$ and $\E[\log|A_n|] > -\infty.$
\end{enumerate}
Then we have for all large $n\in \N$ that
\begin{align} \label{2.1}
\E\left[\left|\tau_n(A_r(\alpha,\beta))-\frac{\beta-\alpha}{2\pi}\right|\right] \leq C_r \left[\frac{1}{n}\left(\frac{1}{t}\log \sum_{k=0}^n \E[|A_k|^t] - \frac{1}{2}\E[\log|A_0A_n|] \right)\right]^{1/2},
\end{align}
where
$$ C_r := \sqrt{\frac{2\pi}{\mathbf{k}}}+\frac{2}{1-r} \quad\mbox{with}\quad \mathbf{k}:=\sum_{k=0}^{\infty}\frac{(-1)^k}{(2k+1)^2}$$
being Catalan's constant.
\end{theorem}

Introducing uniform bounds, we obtain the rates of convergence for the expected discrepancy as $n\to\infty.$
\begin{corollary} \label{cor2.2}
Let $P_n(z)=\sum_{k=0}^n A_{k,n} z^k,\ n\in\N,$ be a sequence of random polynomials. If
$$M := \sup\{\E[|A_{k,n}|^t]\ \vert \ k=0,\ldots,n,\ n\in\N\} < \infty$$
and
$$L :=  \inf\{\E[\log|A_{k,n}|]\ \vert \ k=0\,\&\,n,\ n\in\N\} > - \infty,$$
then
\begin{align*}
\E\left[\left|\tau_n(A_r(\alpha,\beta))-\frac{\beta-\alpha}{2\pi}\right|\right] \leq C_r \left[\frac{1}{n}\left(\frac{\log (n+1)+\log M}{t} - L \right)\right]^{1/2} = O\left(\sqrt{\frac{\log{n}}{n}}\right)
\end{align*}
as $n\to\infty.$
\end{corollary}

The arguments of \cite{PS} now give quantitative results about the expected number of zeros of random polynomials in various sets. We first consider sets separated from $\T.$
\begin{proposition} \label{prop2.3}
Let $E\subset \C$ be a compact set such that $E\cap \T=\emptyset$, and set $d:=\text{dist}(E,\T)$.  If $P_n$ is as in Theorem \ref{thm2.1}, then the expected number of its zeros in E satisfies
$$\E[n\tau_n(E)]\leq \frac{d+1}{d}\left( \frac{2}{t}\log \left(\sum_{k=0}^n \E[|A_k|^t]\right) - \E[\log|A_0A_n|] \right).$$
\end{proposition}

Just as in \cite{PS}, the following proposition gives a bound on the number of zeros in sets that have non-tangential contact with $\T$.
\begin{proposition} \label{prop2.4}
If $E$ is a polygon inscribed in $\T$, and the sequence $\{P_n\}_{n=1}^{\infty}$ is as in Corollary \ref{cor2.2}, then the expected number of zeros of $P_n$ in $E$ satisfies
$$\E[n\tau_n(E)]=O\left(\sqrt{n\log n}\right)  \ \ \text{as} \ \   n\rightarrow \infty.$$
\end{proposition}

Finally, if an open set insects $\T$, then it must carry a positive fraction of zeros according to the normalized arclength measure on $\T$. This is illustrated below for the disks $D_r(w)=\{z\in \C:|z-w|<r\}$, $w\in \T$.
\begin{proposition} \label{prop2.5}
If $w\in \T$ and $r<2$, and the sequence $\{P_n\}_{n=1}^{\infty}$ is as in Corollary \ref{cor2.2}, then the expected number of zeros of $P_n$ in $D_r(w)$ satisfies
$$\E[n\tau(D_r(w))]=\frac{2\arcsin (r/2)}{\pi} \ n+O\left(\sqrt{n\log n}\right) \ \ \text{as} \ \ n\rightarrow \infty.$$
\end{proposition}

\section{Random Polynomials Spanned by General Bases}

We now analyze the behavior of random polynomials spanned by general bases.  Throughout this section, let $B_k(z)=\sum_{j=0}^k b_{j,k}z^j$, where $b_{j,k}\in\C$ for all $j$ and $k$, and $b_{k,k}\neq 0$ for all $k$, be a polynomial basis, i.e. a linearly independent set of polynomials. Observe that $\text{deg}\  B_k=k$ for all $k\in\N\cup\{0\}.$ We study the zero distribution of random polynomials
$$ P_n(z) = \sum_{k=0}^n A_k B_k(z).$$
Throughout this section, we  assume that
\begin{align} \label{3.1}
\limsup_{k\to\infty} \|B_k\|_{\infty}^{1/k}\leq 1 \quad \mbox{and} \quad \lim_{k \rightarrow \infty} |b_{k,k}|^{1/k}=1.
\end{align}
It is well known that $\|B_k\|_{\infty} \ge |b_{k,k}|$ holds for all polynomials, so that \eqref{3.1}  in fact implies $\lim_{k\to\infty} \|B_k\|_{\infty}^{1/k} = 1$. Conditions \eqref{3.1} hold for many standard bases used for representing analytic functions in the unit disk, e.g., for various sequences of orthogonal polynomials (cf. Stahl and Totik \cite{ST}). In the latter case, random polynomials spanned by such bases are called random orthogonal polynomials. Their asymptotic zero distribution was recently studied in a series of papers by Shiffman and Zelditch \cite{SZ}, Bloom \cite{Bl}, and others.

Our main result of this section is the following:
\begin{theorem} \label{thm3.1}
For $P_n(z)=\sum_{k=0}^n A_k B_k(z),$ let $\{A_k\}_{k=0}^n$ be random variables satisfying $\E[|A_k|^t]<\infty,\ k=0,\ldots,n,$ for a fixed $t>0,$ and set $D_n := A_n b_{n,n} \sum_{k=0}^n A_k b_{0,k}$. If $\E[\log|D_n|] > -\infty$ then we have for all large $n\in \N$ that
\begin{align} \label{3.2}
&\E\left[\left|\tau_n(A_r(\alpha,\beta))-\frac{\beta-\alpha}{2\pi}\right|\right] \\ \nonumber &\le C_r \left[\frac{1}{n}\left(\frac{1}{t}\log \left(\sum_{k=0}^n \E[|A_k|^t]\right) + \log \max_{0 \le k \le n} \|B_k\|_{\infty} - \frac{1}{2} \E[\log|D_n|]\right)\right]^{1/2},
\end{align}
where
$$ C_r = \sqrt{\frac{2\pi}{\mathbf{k}}}+\frac{2}{1-r}. $$
In particular, if $\E[\log|A_n|] > -\infty$ and $\E[\log|A_0+z|] \ge L > -\infty$ for all $z\in\C,$ then
\begin{align} \label{3.3}
\E[\log|D_n|] \ge \log|b_{0,0}b_{n,n}| + \E[\log|A_n|] + L > -\infty,
\end{align}
and \eqref{3.2} holds.
\end{theorem}

An example of a typical basis satisfying \eqref{3.1} is given below by orthonormal polynomials on the unit circle. We apply Theorem \ref{thm3.1} to obtain a quantitative result on the zero distribution of random orthogonal polynomials.

\begin{corollary} \label{cor3.2}
Let $P_n(z)=\sum_{k=0}^n A_{k,n} B_k(z),\ n\in\N,$ be a sequence of random orthogonal polynomials. Suppose that the following uniform estimates for the coefficients hold true:
\begin{align} \label{3.4}
\sup\{\E[|A_{k,n}|^t]\ \vert \ k=0,\ldots,n;\ n\in\N\} < \infty,\quad t>0,
\end{align}
and
\begin{align} \label{3.5}
\min\left(\inf_{n\in\N} \E[\log|A_{n,n}|], \inf_{n\in\N,z\in\C} \E[\log|A_{0,n}+z|]\right) > - \infty.
\end{align}
If the basis polynomials $B_k$ are orthonormal with respect to a positive Borel measure $\mu$ supported on $\T=\{e^{i\theta}:0\le\theta<2\pi\}$, such that the Radon-Nikodym derivative $d\mu/d\theta>0$ for almost every $\theta\in[0,2\pi),$ then \eqref{3.1} is satisfied and
\begin{align} \label{3.6}
\lim_{n\to\infty} \E\left[\left|\tau_n(A_r(\alpha,\beta))-\frac{\beta-\alpha}{2\pi}\right|\right] = 0.
\end{align}
Furthermore, if $d\mu(\theta)=w(\theta)\,d\theta$, where $w(\theta)\ge c > 0,\ \theta\in[0,2\pi),$ then
\begin{align} \label{3.7}
\E\left[\left|\tau_n(A_r(\alpha,\beta))-\frac{\beta-\alpha}{2\pi}\right|\right] = O\left(\sqrt{\frac{\log{n}}{n}}\right) \quad n\to\infty.
\end{align}
\end{corollary}

It is clear that if the coefficients have identical distributions, then all uniform bounds in \eqref{3.4} and \eqref{3.5} reduce to those on the single coefficient $A_0.$ One can relax conditions on the orthogonality measure $\mu$ while preserving the results, e.g., one can show that \eqref{3.7} also holds for the generalized Jacobi weights of the form $w(\theta)=v(\theta)\prod_{j=1}^J |\theta-\theta_j|^{\alpha_j},$ where $v(\theta)\ge c > 0,\ \theta\in[0,2\pi).$ Note that the analogs of Propositions \ref{prop2.4}-\ref{prop2.5} for the random orthogonal polynomials follow from \eqref{3.7}.

\section{Discrete Random Coefficients}

Let $A_0, A_1,\dots$ be independent and identically distributed (iid) complex discrete random variables. We show that one can extend the ideas of \cite{PS} and prove essentially the same results in the discrete case. Furthermore, since any real random variable is the limit of an increasing sequence of discrete random variables, we are able to extend the arguments to arbitrary random variables. We assume as before that $\E[|A_0|^t]=\mu<\infty$ for a fixed real $t>0$.

\begin{proposition} \label{prop4.1}
Let $A_0, A_1,\dots$ be iid complex random variables, and let $Y_n:=\displaystyle\max_{0\leq k\leq n}|A_k|$. If $\mu:=\E[|A_0|^t]<\infty$, where $t>0$, then $$\E[\log Y_n] \leq \frac{\log (n+1)+\log \mu}{t}.$$
\end{proposition}
This result provides an immediate extension of Theorem 3.3 of \cite{PS} to arbitrary random variables (satisfying the moment assumption) by following the same proof.
Indeed, we have that
\begin{align*}
\E[\log\|P_n\|_{\infty}]&=\E\left[\log\left( \sup_{z\in\T} \left|\sum_{k=0}^nA_kz^k\right|\right)\right]
= \E\left[\log \left(\sum_{k=0}^n|A_k|\right)\right] \\
&\leq \E\left[\log \left((n+1) \max_{0\leq k\leq n} |A_k|\right)\right] = \log (n+1) +\E[\log Y_n].
\end{align*}
Thus referring to the proof of Theorem 3.3 of \cite{PS} and using our bound of $\E[\log Y_n]$ gives the result.

\section{Dependent Coefficients}

We generalize Theorem 3.7 of \cite{PS} in this section, replacing the requirement that the first and the second
moments of the absolute values of all coefficients be equal with the requirement they be uniformly bounded.
More precisely, we assume that
\begin{equation}\label{5.1}
\sup_{k} \E[|A_k|]=:M<\infty \ \ \text{and} \ \ \sup_k \text{Var}[|A_k|]=:S^2<\infty.
\end{equation}
Following the ideas of Arnold and Groeneveld \cite{AG} (see also \cite{DN}), we show that
\begin{proposition} \label{prop5.1}
If \eqref{5.1} is satisfied, then we have for $Y_n=\max_{0\le k\le n} |A_k|$ that
$$\E[Y_n] = O(\sqrt{n}) \quad \mbox{as }n\to\infty.$$
\end{proposition}

An analog of the result from \cite{PS} is obtained along the same lines as before.
\begin{theorem} \label{thm5.2}
If the (possibly dependent) coefficients of $P_n$  satisfy \eqref{5.1} as well as $\E[\log|A_0|] > -\infty$ and $\E[\log|A_n|] > -\infty$, then
\begin{align*}
\E \left[\left|\tau_n(A_r(\alpha,\beta))-\frac{\beta-\alpha}{2\pi}\right|\right]
\leq C_r \sqrt{\frac{\frac{3}{2}\log(n+1)-\frac{1}{2}\E[\log|A_0|] - \frac{1}{2}\E[\log|A_n|] + O(1)}{n}}
\end{align*}
as $n\rightarrow\infty$.
\end{theorem}

Clearly, this result has more restrictive assumptions than Theorem \ref{thm2.1}.

\section{Proofs}

\subsection{Proofs for Section 2}

Define the logarithmic Mahler measure (logarithm of geometric mean) of $P_n$ by
$$m(P_n)=\frac{1}{2\pi}\int_0^{2\pi}\log|P_n(e^{i\theta})|d\theta.$$
It is immediate to see that $m(P_n) \le \log\|P_n\|_{\infty}.$

The majority of our results are obtained with help of the following modified version of the discrepancy theorem due to Erd\H{o}s and Tur\'an (cf. Proposition 2.1 of \cite{PS}):

\begin{lemma}\label{lem6.1}
Let $P_n(z)=\sum_{k=0}^n c_k z^k,\ c_k\in\C$, and assume $c_0c_n\neq 0.$ For any $r\in(0,1)$ and $0\le \alpha < \beta < 2\pi,$ we have
\begin{align} \label{6.1}
\left| \tau_n\left(A_r(\alpha,\beta)\right) - \frac{\beta-\alpha}{2\pi}\right| &\leq
\sqrt{\frac{2\pi}{{\bf k}}} \sqrt{\frac{1}{n}\,\log\frac{\|P_n\|_{\infty}}{\sqrt{|c_0c_n|}}} \\ \nonumber &+ \frac{2}{n(1-r)} \, m\left(\frac{P_n}{\sqrt{|c_0c_n|}}\right),
\end{align}
where $\mathbf{k}=\sum_{k=0}^{\infty}(-1)^k/(2k+1)^2$ is Catalan's constant.
\end{lemma}
This estimate shows how close the zero counting measure $\tau_n$ is to $\mu_{\T}.$

The following lemma is used several times below.

\begin{lemma} \label{lem6.2}
If $A_k,\ k=0,\ldots,n,$ are independent complex random variables satisfying $\E[|A_k|^t]<\infty,\ k=0,\ldots,n,$ for a fixed $t>0,$ then
\begin{align} \label{6.2}
\E\[\log\sum_{k=0}^n |A_k|\] \le \frac{1}{t}\log \left(\sum_{k=0}^n \E[|A_k|^t]\right).
\end{align}
\end{lemma}

\begin{proof}
We first observe an elementary inequality. If $x_i \ge 0,\ i=0,\ldots,n,$ and $\sum_{i=0}^n x_i = 1,$ then for any $t\in(0,1)$ we have that
$$ \sum_{i=0}^n (x_i)^t \ge \sum_{i=0}^n x_i = 1.$$
Applying this inequality with $x_i = |A_i|/\sum_{k=0}^n |A_k|,$ we obtain that
\begin{align*}
\left(\sum_{k=0}^n |A_k|\right)^t \le \sum_{k=0}^n |A_k|^t
\end{align*}
and
\begin{align} \label{6.3}
\E\[\log\sum_{k=0}^n |A_k|\] \le \frac{1}{t}\E\[\log\left(\sum_{k=0}^n |A_k|^t\right)\].
\end{align}
Jensen's inequality and linearity of expectation now give that
\begin{align*}
\E\[\log\sum_{k=0}^n |A_k|\] &\le \frac{1}{t}\log \E\[\sum_{k=0}^n |A_k|^t\] = \frac{1}{t}\log \left(\sum_{k=0}^n \E[|A_k|^t]\right).
\end{align*}
\end{proof}

\begin{proof}[Proof of Theorem \ref{thm2.1}]
Note that $m(Q_n) \le \log \|Q_n\|_{\infty}$  for all polynomials $Q_n$. Hence \eqref{6.1} and Jensen's inequality imply that
\begin{align*}
\E\left[\left| \tau_n\left(A_r(\alpha,\beta)\right) - \frac{\beta-\alpha}{2\pi}\right|\right] &\leq
\sqrt{\frac{2\pi}{{\bf k}}} \sqrt{\frac{1}{n}\, \E\left[ \log\frac{\|P_n\|_{\infty}}{\sqrt{|A_0 A_n|}} \right]} + \frac{2}{n(1-r)} \, \E\left[\log\frac{\|P_n\|_{\infty}}{\sqrt{|A_0 A_n|}}\right] \\ &\leq
C_r \sqrt{\frac{1}{n}\, \E\left[ \log\frac{\|P_n\|_{\infty}}{\sqrt{|A_0 A_n|}} \right]},
\end{align*}
where the last inequality holds for all sufficiently large $n\in\N.$ Since $\|P_n\|_{\infty} \le \sum_{k=0}^n |A_k|,$ we use the linearity of expectation and \eqref{6.2} to estimate
\begin{align*}
\E\left[ \log\frac{\|P_n\|_{\infty}}{\sqrt{|A_0 A_n|}} \right] &\le \E\left[\log \sum_{k=0}^n |A_k|\right] - \frac{1}{2}\E[\log|A_0A_n|] \\ &\le \frac{1}{t}\log \left(\sum_{k=0}^n \E[|A_k|^t]\right) - \frac{1}{2}\E[\log|A_0A_n|].
\end{align*}
The latter upper bound is finite by our assumptions.
\end{proof}

\begin{proof}[Proof of Corollary \ref{cor2.2}]
The result follows immediately upon using the uniform bounds $M$ and $L$ in estimate \eqref{2.1}.
\end{proof}

\begin{proof}[Proof of Proposition \ref{prop2.3}]
In was shown in \cite{PS} (see (5.3) in that paper) that
$$\tau_n(\C\setminus A_r(0,2\pi))\leq \frac{2}{n(1-r)}m\left(\frac{P_n}{\sqrt{|A_0 A_n|}}\right).$$
Since $m(Q_n) \le \log \|Q_n\|_{\infty}$  for all polynomials $Q_n$, it follows that
$$\tau_n(\C\setminus A_r(0,2\pi))\leq \frac{2}{n(1-r)}\log\left(\frac{\|P_n\|_{\infty}}{\sqrt{|A_0 A_n|}}\right).$$
Note that for $r=1/(\text{dist}(E,\T)+1)$, we have $E\subset \C\setminus A_r(0,2\pi)$. Estimating $\|P_n\|_{\infty}$ as in the proof of Theorem \ref{thm2.1}, we obtain that
\begin{align*}
\E[n\tau_n(E)] & \leq \frac{2}{1-r}\E\left[\log\left(\frac{\|P_n\|_{\infty}}{\sqrt{|A_0 A_n|}}\right)\right] \\
&\leq \frac{2}{1-r}\left(  \frac{1}{t}\log \left(\sum_{k=0}^n \E[|A_k|^t]\right) - \frac{1}{2}\E[\log|A_0A_n|]  \right) \\
&=\frac{d+1}{d}\left(  \frac{2}{t}\log \left(\sum_{k=0}^n \E[|A_k|^t]\right) -\E[\log|A_0A_n|]  \right).
\end{align*}
\end{proof}

\begin{proof}[Proof of Proposition \ref{prop2.4}]
The proof of this proposition proceeds in the same manner as the proof of Proposition 3.5 in \cite{PS} by using our Corollary \ref{cor2.2} along with Proposition \ref{prop2.3} .
\end{proof}

\begin{proof}[Proof of Proposition \ref{prop2.5}]
As in the previous proof, this result follows in direct parallel to the proof of Proposition 3.6 of \cite{PS} while taking into account our bound in Proposition \ref{prop2.4}.
\end{proof}

\subsection{Proofs for Section 3}

\begin{proof}[Proof of Theorem \ref{thm3.1}]
We proceed with an argument similar to the proof of Theorem \ref{thm2.1}. Note that the leading coefficient of $P_n$ is $A_nb_{n,n}$, and its constant term is $\sum_{k=0}^n A_k b_{0,k}$. Using the fact $m(Q_n) \le \log \|Q_n\|_{\infty}$  for all polynomials $Q_n$, we apply \eqref{6.1} and Jensen's inequality to obtain
\begin{align*}
\E\left[\left| \tau_n\left(A_r(\alpha,\beta)\right) - \frac{\beta-\alpha}{2\pi}\right|\right] &\leq
\sqrt{\frac{2\pi}{{\bf k}}} \sqrt{\frac{1}{n}\, \E\left[ \log\frac{\|P_n\|_{\infty}}{\sqrt{|D_n|}} \right]} + \frac{2}{n(1-r)} \, \E\left[\log\frac{\|P_n\|_{\infty}}{\sqrt{|D_n|}}\right] \\ &\leq C_r \sqrt{\frac{1}{n}\, \E\left[ \log\frac{\|P_n\|_{\infty}}{\sqrt{|D_n|}} \right]}
\end{align*}
for all sufficiently large $n\in\N.$ It is clear that
$$\|P_n\|_{\infty} \le \max_{0 \le k \le n} \|B_k\|_{\infty} \sum_{k=0}^n |A_k|.$$
Hence \eqref{6.1} yields
\begin{align*}
\E\left[ \log\frac{\|P_n\|_{\infty}}{\sqrt{|D_n|}} \right] &\le \E\left[\log \sum_{k=0}^n |A_k|\right] + \log \max_{0 \le k \le n} \|B_k\|_{\infty} - \frac{1}{2} \E[\log|D_n|] \\ &\le \frac{1}{t}\log \left(\sum_{k=0}^n \E[|A_k|^t]\right) + \log \max_{0 \le k \le n} \|B_k\|_{\infty} - \frac{1}{2} \E[\log|D_n|].
\end{align*}
Thus \eqref{3.2} follows as a combination of the above estimates.

We now proceed to the lower bound for the expectation of $\log|D_n|$ in \eqref{3.3} by estimating that
\begin{align*}
\E[\log|D_n|] &= \E\left[\log \left|A_n b_{n,n} \sum_{k=0}^n A_k b_{0,k}\right|\right]\\
&= \E[\log|A_n|] + \log|b_{n,n}| + \E\left[\log \left|\sum_{k=0}^n A_k b_{0,k}\right|\right]\\
&= \E[\log|A_n|] + \log|b_{n,n}| + \log|b_{0,0}| + \E\left[\log \left| A_0 + \sum_{k=1}^n A_k \frac{b_{0,k}}{b_{0,0}}\right|\right]\\
&\ge \log|b_{0,0}b_{n,n}| + \E[\log|A_n|] + L,
\end{align*}
where we used that $b_{0,0} \neq 0$ and $\E[\log|A_0+z|] \ge L$ for all $z\in\C.$
\end{proof}

\begin{proof}[Proof of Corollary \ref{cor3.2}]
We apply \eqref{3.2} with \eqref{3.3}. The uniform bounds on the expectations for the coefficients immediately give that
\begin{align*}
\frac{1}{tn}\log \left(\sum_{k=0}^n \E[|A_{k,n}|^t]\right) = O\left(\frac{\log n}{n}\right) \quad\mbox{and}\quad \frac{1}{2n} \E[\log|D_n|] \ge \frac{1}{n} \log|b_{n,n}| + O\left(\frac{1}{n}\right).
\end{align*}
The assumption $d\mu/d\theta>0$ for a.e. $\theta$ implies \eqref{3.1}, see Corollary 4.1.2 of \cite{ST}, which in turn gives that
$$\lim_{n\to\infty} \frac{1}{n} \log|b_{n,n}| = \lim_{n\to\infty} \frac{1}{n} \log \max_{0 \le k \le n} \|B_k\|_{\infty} = 0.$$
Hence \eqref{3.6} follows from \eqref{3.2}. Recall that the leading coefficient $b_{n,n}$ of the orthonormal polynomial $B_n$ gives the solution of the following extremal problem \cite{ST}:
$$ |b_{n,n}|^{-2} = \inf\left\{ \int |Q_n|^2\,d\mu \ : \ Q_n \mbox{ is a monic polynomial of degree }n \right\}.$$
Using $Q_n(z)=z^n,$ we obtain that
$$ |b_{n,n}| \ge \left(\mu(\T)\right)^{-1/2} \mbox{  and   } \frac{1}{n} \log|b_{n,n}| \ge - \frac{1}{2n}\log\mu(\T). $$
We now show that $ \log \|B_n\|_{\infty} = O(\log n)$ as $n\to\infty,$ provided $d\mu(\theta)=w(\theta)\,d\theta$ with $w(\theta)\ge c > 0,\ \theta\in[0,2\pi).$ Indeed, the Cauchy-Schwarz inequality gives for the orthonormal polynomial $B_n(z)=\sum_{k=0}^n b_{k,n} z^k$ that
\begin{align*}
\|B_n\|_{\infty} &\le \sum_{k=0}^n |b_{k,n}| \le \sqrt{n+1} \left(\sum_{k=0}^n |b_{k,n}|^2\right)^{1/2} = \sqrt{n+1} \left(\frac{1}{2\pi} \int_0^{2\pi} |B_n(e^{i\theta})|^2\,d\theta \right)^{1/2} \\ &\le \sqrt{\frac{n+1}{2\pi c}} \left(\int_0^{2\pi} |B_n(e^{i\theta})|^2 w(\theta)\,d\theta \right)^{1/2} = \sqrt{\frac{n+1}{2\pi c}}.
\end{align*}
This estimate completes the proof of \eqref{3.7}.
\end{proof}

\subsection{Proofs for Section 4}

\begin{proof}[Proof of Proposition \ref{prop4.1}]
Assume that the discrete random variable $|A_0|$ takes values $\{x_k\}_{k=1}^{\infty}$ that are arranged in the increasing order, and note that the range of values for $Y_n$ is the same. Let $a_k=\P(Y_n\leq x_k)$ and $b_k=\P(|A_0|\leq x_k)$, where $k\in\N$. It is clear that $\P(Y_n=x_k) = a_k-a_{k-1}$ and $\P(|A_0|=x_k)=b_k-b_{k-1},\ k\in\N.$ Since the $A_k$'s are independent and identically distributed, we have that
\begin{align*}
a_k &= \P(Y_n\leq x_k)= \P(|A_0|\leq x_k, |A_1|\leq x_k,\dots, |A_n|\leq x_k)\\
&=\P(|A_0|\leq x_k) \P( |A_1|\leq x_k) \cdots \P( |A_n|\leq x_k) = [\P(|A_0|\leq x_k)]^{n+1}=b_k^{n+1}
\end{align*}
holds for all $k\in\N.$ Thus
\begin{align*}
\E[Y_n^t]:=&\sum_{k=1}^{\infty} x_k^t \ \P(Y_n=x_k) =\sum_{k=1}^{\infty} x_k^t \ [a_k-a_{k-1}] \\
=&\sum_{k=1}^{\infty} x_k^t \ [b_k^{n+1}-b_{k-1}^{n+1}]
= \sum_{k=1}^{\infty} x_k^t \ [b_k -b_{k-1}][b_k^n +b_k^{n-1}b_{k-1} + \dots +b_{k-1}^n] \\
\leq & \sum_{k=1}^{\infty} x_k^t \  [b_k-b_{k-1}](n+1)b_k^n \leq (n+1) \sum_{k=1}^{\infty} x_k^t \ \P(|A_0| = x_k) \\
=&(n+1) \ \E[|A_0|^t].
\end{align*}
By Jensen's inequality and the previous estimate, we have
\begin{align*}
\E[\log Y_n]=\E\left[\frac{1}{t}\log Y_n^t\right]\leq &\frac{1}{t} \log \E [Y_n^t] \\
\leq & \frac{1}{t} ( \log( (n+1) \ \E[|A_0|^t]) \\
=& \frac{1}{t}(\log (n+1)+\log \mu).
\end{align*}

 We now show that this argument can be extended to arbitrary random variables $\{|C_k|\}_{k=0}^n$. Consider the increasing sequences of simple (discrete) random variables $\{|A_{k,i}|\}_{i=1}^{\infty}$ such that $\lim_{i\rightarrow \infty}|A_{k,i}|=|C_k|,\ k=0,\dots, n$. For $Y_{n,i}=\max_{0\leq k\leq n}|A_{k,i}|$ and $Z_n=\max_{0\leq k \leq n}|C_{k}|$, one can see that
$$\lim_{i\rightarrow \infty}Y_{n,i}^t=Z_n^t \ \ \ \text{and} \ \ \ \ \lim_{i\rightarrow \infty} |A_{0,i}|^t=|C_0|^t,$$
where $t>0$. Moreover, the sequence of simple random variables $Y_{n,i}^t$ is increasing to $Z_n^t$, so that the Monotone Convergence Theorem gives
$$\lim_{i\rightarrow \infty} \E[Y_{n,i}^t]=\E[Z_n^t].$$
Using the already proven result for discrete random variables and passing to the limit as $i\to\infty$, we obtain that
$$\E[Z_{n}^t]\leq (n+1)\E[|C_{0}|^t].$$
Hence Jensen's inequality yields
$$\E[\log Z_n]\leq \frac{1}{t}(\log (n+1)+\log \E[|C_0|^t]),$$
as before.
\end{proof}

\subsection{Proofs for Section 5}

The following lemma is due to Arnold and Groeneveld \cite{AG}, and is also found in \cite[p. 110]{DN}. We prove it in our setting for completeness.
\begin{lemma} Let $X_i,\ i=0,1,\dots, n,$ be possibly dependent random variables with $\E[X_i]=\mu_i$ and $\text{Var}[X_i]=\sigma_i^2$. Then for any real constants $c_i$, the ordered random variables $X_{0:n}\leq X_{1:n}\leq \cdots \leq X_{n:n}$ satisfy
$$\left|\E\left[ \sum_{i=0}^n c_i(X_{i:n}-\bar{\mu})\right]\right|\leq \left(\sum_{i=0}^n (c_i-\bar{c})^2\sum_{i=0}^n[(\mu_i-\bar{\mu})^2+\sigma_i^2]\right)^{1/2},$$
where $\bar{c}=n^{-1}\sum_{i=0}^n c_i$, $\bar{\mu}=n^{-1}\sum_{i=0}^n \mu_{i:n}=n^{-1}\sum_{i=0}^n \mu_i$, and $\mu_{i:n}=\E[X_{i:n}]$.
\end{lemma}
\begin{proof}
We use the Cauchy-Schwartz inequality in the following estimate:
\begin{align*}
\left|\sum_{i=0}^n c_i(X_{i:n}-\bar{\mu})\right|=& \left|\sum_{i=0}^n (c_i-\bar{c})(X_{i:n}-\bar{\mu})\right| \\
\leq & \ \left[\sum_{i=0}^n (c_i-\bar{c})^2\sum_{i=0}^n(X_{i:n}-\bar{\mu})^2\right]^{1/2}. \\
\end{align*}
Observe that $|\E(Y)|\leq \E(|Y|)$ for any random variable $Y$, and that $\E(Z^{1/2})\leq [\E(Z)]^{1/2}$ for $Z\geq 0$ by Jensen's inequality.  Applying these facts while taking the expectation of the previous inequality gives
\begin{align*}
\left|\E\left[\sum_{i=0}^n c_i(X_{i:n}-\bar{\mu})\right]\right|\leq & \  \left[\sum_{i=0}^n (c_i-\bar{c})^2\right]^{1/2} \ \left[\E\left[\sum_{i=0}^n(X_{i:n}-\bar{\mu})^2\right]\right]^{1/2} \\
=&\left[\sum_{i=0}^n (c_i-\bar{c})^2\right]^{1/2} \ \left[\sum_{i=0}^n\E[X_{i:n}^2]-2\E[X_{i:n}]\bar{\mu}-\bar{\mu}^2)\right]^{1/2} \\
= & \ \left[\sum_{i=0}^n (c_i-\bar{c})^2\right]^{1/2} \ \left[\sum_{i=0}^n\sigma^2_i+(\mu_i-\bar{\mu})^2\right]^{1/2}.
\end{align*}
\end{proof}

\begin{proof}[Proof of Proposition \ref{prop5.1}]
To obtain bounds for $\E[Y_n]=\mu_{n:n}=\E[A_{n:n}]$, we apply the previous lemma while choosing $c_0=c_1= \dots= c_{n-1}=0$ and $c_n=1$. This yields
\begin{align*}
\E[A_{n:n}]-\bar{\mu}&\leq \left((n\bar{c}^2+ (1-\bar{c})^2) \sum_{i=0}^n(\mu_i^2-2\mu_i\bar{\mu}+\bar{\mu}^2+\sigma^2_i)\right)^{1/2}\\
&=\left(\left(\frac{n}{(n+1)^2}+ \left(1-\frac{1}{n+1}\right)^2\right) \sum_{i=0}^n(\mu_i^2-2\mu_i\bar{\mu}+\bar{\mu}^2+\sigma^2_i)\right)^{1/2}\\
&\leq \left(\sum_{i=0}^n (M^2+2M^2+M^2+S^2)\right)^{1/2} \\
&=(4M^2+S^2)^{1/2}(n+1)^{1/2}.
\end{align*}
It follows that
\begin{align*}
\E[Y_n]=\E[A_{n:n}]&\leq \bar{\mu} + (4M^2+S^2)^{1/2}(n+1)^{1/2} \\
&\leq M + (4M^2+S^2)^{1/2}(n+1)^{1/2}.
\end{align*}
\end{proof}

\begin{proof}[Proof of Theorem \ref{thm5.2}]
As in the proof of Theorem \ref{thm2.1}, we apply \eqref{6.1} and Jensen's inequality to obtain for all sufficiently large $n\in\N$ the following
\begin{align*}
\E\left[\left| \tau_n\left(A_r(\alpha,\beta)\right) - \frac{\beta-\alpha}{2\pi}\right|\right] &\leq C_r \sqrt{\frac{1}{n}\, \E\left[ \log\frac{\|P_n\|_{\infty}}{\sqrt{|A_0 A_n|}} \right]} \\
&=C_r \sqrt{\frac{ \E[\log \|P_n\|_{\infty}] -\frac{1}{2}\E[\log|A_0|]-\frac{1}{2}\E[\log |A_n|]}{n}}. \\
\end{align*}
Observe that
\begin{align*}\|P_n\|_{\infty} & =\sup_{\T}\left|\sum_{k=0}^n A_kz^k\right|
\leq \sum_{k=0}^n |A_k|\leq (n+1)\max_{0\leq k \leq n}|A_k|
=(n+1) Y_n.
\end{align*}
Taking the logarithm and then the expectation of the above yields
\begin{align*}
\E[\log \|P_n\|_{\infty}]&\leq \E[\log (n+1) + \log Y_n] \\
&=\log (n+1) + \E[\log Y_n] \\
&\leq \log (n+1) + \log\E[ Y_n], \\
\end{align*}
where the last inequality follows from Jensen's inequality.  As $n\rightarrow \infty$, applying proposition \ref{prop5.1} gives
\begin{align*}
\log (n+1) + \log\E[ Y_n]
&\leq \log(n+1) + \log O(\sqrt{n}) \\
&=\log(n+1) + \frac{1}{2}\log n +O(1) \\
&<\frac{3}{2} \log (n+1) + O(1).
\end{align*}
Combining these bounds gives the result of Theorem \ref{thm5.2}.
\end{proof}

\textbf{Acknowledgements.}  Research of I. E. Pritsker was partially supported by the National Security Agency (grant H98230-12-1-0227), and by the AT\&T Professorship. Research of A. M. Yeager was partially supported by the Vaughn Foundation and by the Jobe scholarship from the Department of Mathematics at Oklahoma State University, and it is a portion of his work towards a PhD degree.

\end{document}